\documentclass[11pt,reqno]{amsart}

\usepackage[foot]{amsaddr}
\usepackage{enumerate}
\usepackage{amsfonts, amsmath, amssymb, amscd, amsthm, bm}

\usepackage{enumerate}
\usepackage{url}
\usepackage[linktocpage=true,colorlinks,citecolor=magenta,linkcolor=blue,urlcolor=magenta]{hyperref}

\DeclareMathOperator{\Tr}{Tr}

 \newtheorem{thm}{Theorem}[section]
 \newtheorem{cor}[thm]{Corollary}
 
 \newtheorem{lem}[thm]{Lemma}

 \theoremstyle{definition}
 
 \newtheorem{rem}[thm]{Remark}
 
 \newtheorem*{ack}{Acknowledgments}

 \numberwithin{equation}{section}
\numberwithin{equation}{section}

\newcounter{rom}
\renewcommand{\therom}{(\roman{rom})}
{\end{list}}

\title[Sobolev inequalities]{Sobolev inequalities in manifolds with nonnegative intermediate Ricci curvature}
\begin{document}

\author{Hui Ma}
\address{Department of Mathematical Sciences, Tsinghua University,
Beijing 100084, P.R. China} 
\email{ma-h@mail.tsinghua.edu.cn}
\author{Jing Wu*}
\email{wu-j19@mails.tsinghua.edu.cn}

\begin{abstract}
We prove Michael-Simon type Sobolev inequalities for $n$-dimensional submanifolds in $(n+m)$-dimensional Riemannian manifolds with nonnegative $k$-th intermediate Ricci curvature by using the Alexandrov-Bakelman-Pucci method. Here $k=\min(n-1,m-1)$.
These inequalities extend Brendle's Michael-Simon type Sobolev inequalities on Riemannian manifolds with nonnegative sectional curvature \cite{Brendle22}  and Dong-Lin-Lu's   Michael-Simon type Sobolev inequalities on Riemannian manifolds with asymptotically nonnegative sectional curvature  \cite{DLL22} to the $k$-Ricci curvature setting. In particular, a simple application of these inequalities gives rise to some isoperimetric inequalities for minimal submanifolds in Riemannian manifolds.
\end{abstract}

\keywords {Isoperimetric Inequality, Michael-Simon Inequality, Intermediate Ricci Curvature, Minimal Submanifold}

\subjclass[2020]{53C40, 53C21.}

\maketitle

\section{Introduction}\label{sec1}
The classical isoperimetric problem is to find the largest possible area for a planar domain with given perimeter. The attendant isoperimetric inequality also has a long history and has been developed in many different settings. One of the intriguing directions is to prove isoperimetric inequality for minimal surfaces (cf. \cite{Carleman21, Reid59, OS75, Feinberg77, Chavel78, LSY84}). When we turn to minimal submanifolds in $\mathbb{R}^{n+1}$, the isoperimetric inequality is closely related to the famous Michael-Simon Sobolev inequality (cf. \cite{MS73}). In a recent breakthrough, inspired by the Alexandrov-Bakelman-Pucci technique in the proof of isoperimetric inequality (cf. \cite{Cabre08, Trudinger94}), Brendle \cite{Brendle21} proved an elegant Michael-Simon-type inequality. When the codimension is at most $2$, this solves the 
long-standing conjecture that the (sharp) isoperimetric inequality holds for minimal submanifolds in $\mathbb{R}^{n+1}$. Moreover, Brendle \cite{Brendle22} generalized the Michael-Simon type inequality as well as the isoperimetric inequality to minimal submanifolds in Riemannian manifolds with nonnegative sectional curvature.
For recent progress about isoperimetric inequality for minimal submanifolds, we refer to \cite{Choe05, Brendle21, Brendle22} and references therein.

Brendle's work \cite{Brendle22} has been extended to several different curvature settings. For example, Johne \cite{Johne21} considered the case of nonnegative Bakry-Émery Ricci curvature and Dong-Lin-Lu \cite{DLL22} considered the case of asymptotically nonnegative curvature. In this paper, we focus on the intermediate Ricci curvature (or simply $k$-Ricci curvature for the $k$-th intermediate Ricci curvature), which can be regarded as the average of some sectional curvatures. To the best of our knowledge, the notion of $k$-Ricci curvature was introduced by Bishop and Crittenden (cf. \cite{BC64}, p.253). Since the $k$-Ricci curvature interpolates between the sectional curvature and the Ricci curvature, it is natural to consider $k$-Ricci bounds as a weaker curvature condition instead of the sectional curvature bound.
Some early results involving $k$-Ricci bounds were obtained by Galloway \cite{Galloway81}, Wu \cite{Wu87} and Shen \cite{Shen89}, etc.
Recently there has been an increasing interest in the relations between $k$-Ricci bounds and the topology and geometry of manifolds (cf. \cite{XG12, Chahine20, Mouille22, RW22}).
Remark that the link between  intermediate Ricci curvatures and optimal transport was discussed in \cite{KM18, wang22}. Notice that the definition of intermediate Ricci curvature there is different from the above usual one.

One of our main results is the following theorem which extends Theorem 1.4 in \cite{Brendle22} to the $k$-Ricci setting.
\begin{thm}\label{thm1}
Let $M$ be a complete noncompact $(n+m)$-dimensional Riemannian manifold with nonnegative $k$-Ricci curvature, where $k=\min(n-1,m-1)$. Let $\Sigma$ be a compact $n$-dimensional  submanifold of $M$ (possibly with boundary $\partial\Sigma$), and let $f$ be a positive smooth function on $\Sigma$. If $m\ge 2$, then
\begin{equation}
    \int_\Sigma \sqrt{|D^\Sigma f|^2+f^2 |H|^2}+\int_{\partial\Sigma}f \ge n \left(\frac{(n+m)|B^{n+m}|}{m|B^m|}\right)^{\frac{1}{n}}\theta^{\frac{1}{n}}\left(\int_\Sigma f^{\frac{n}{n-1}}\right)^{\frac{n-1}{n}},
\end{equation}
where $\theta$ denotes the asymptotic volume ratio of $M$ and $H$ denotes the mean curvature vector of $\Sigma$.
\end{thm}

By putting $f\equiv1$ and $H\equiv 0$ in Theorem \ref{thm1}, we obtain an isoperimetric inequality.

\begin{cor}
Let $M$ be a complete noncompact $(n+m)$-dimensional Riemannian manifold with nonnegative $k$-Ricci curvature, where $k=\min(n-1,m-1)$. Let $\Sigma$ be a compact minimal $n$-dimensional submanifold of $M$ with boundary $\partial\Sigma$. If $m\ge 2$, then
\begin{equation}
    |\partial\Sigma|\ge n \left( \frac{(n+m)|B^{n+m}|}{m|B|^m}\right)^{\frac{1}{n}}\theta^{\frac{1}{n}}|\Sigma|^{\frac{n-1}{n}},
\end{equation}
where $\theta$ denotes the asymptotic volume ratio of $M$.
In particular, if $m=2$, then
\begin{equation}
    |\partial\Sigma|\ge n |B^n|^{\frac{1}{n}}\theta^{\frac{1}{n}}|\Sigma|^{\frac{n-1}{n}}.
\end{equation}
\end{cor}

\begin{rem}
We should mention that Wang \cite{wang22} recently provided an optimal transport proof of the Michael-Simon inequality using a different definition for the intermediate Ricci curvature.
\end{rem}

Another main result is a generalization of Theorem \ref{thm1}, which extends Theorem 1.5 in \cite{DLL22} to the case that the ambient manifold has asymptotically nonnegative $k$-Ricci curvature.

\begin{thm}\label{thm2}
Let $M$ be a complete noncompact $n+m$ dimensional Riemannian manifold of asymptotically nonnegative $k$-Ricci curvature with respect to a base point $o$ in $M$, where $k=\min(n-1,m-1)$. Let $\Sigma$ be a compact $n$-dimensional submanifold of $M$ (possibly with boundary $\partial\Sigma$), and let $f$ be a positive smooth function on $\Sigma$. If $m\ge 2$, then
\begin{equation}
\begin{aligned}
    &\int_\Sigma \sqrt{|D^\Sigma f|^2+f^2 |H|^2} + 2n b_1\int_\Sigma f + \int_{\partial\Sigma}f \\
    &\ge n \left(\frac{(n+m)|B^{n+m}|}{m|B^m|}\right)^{\frac{1}{n}}\theta_h^{\frac{1}{n}} \left(\frac{1+b_0}{e^{2r_0b_1+b_0}}\right)^{\frac{n+m-1}{n}} \left(\int_\Sigma f^{\frac{n}{n-1}}\right)^{\frac{n-1}{n}},
\end{aligned}
\end{equation}
where $r_0=\max\{d(o,x)|x\in \Sigma\}$, $\theta_h$ denotes the asymptotic volume ratio of $M$ with respect to $h$, $H$ denotes the mean curvature vector of $\Sigma$ and $b_0$, $b_1$ are defined by \eqref{defb0} and \eqref{defb1}.
\end{thm}
In particular, we can also obtain an isoperimetric type inequality by putting $f\equiv 1$ and $H\equiv 0$ in Theorem \ref{thm2}.
\begin{cor}\label{cor:iso_asym}
    Assuming same conditions as in Theorem \ref{thm2}, we have
    \begin{equation}
        \begin{aligned}
        |\partial\Sigma|\ge n|\Sigma|^{\frac{n-1}{n}}\left[ \left(\frac{(n+m)|B^{n+m}|}{m|B^m|}\right)^{\frac{1}{n}}\theta_h^{\frac{1}{n}} \left(\frac{1+b_0}{e^{2r_0b_1+b_0}}\right)^{\frac{n+m-1}{n}} -2 b_1 |\Sigma|^{\frac{1}{n}}\right].
        \end{aligned}
    \end{equation}
\end{cor}

This paper is organized as follows. Section \ref{sec2} contains some basic concepts. In Section \ref{sec3}, we deal with the $k$-Ricci curvature and prove Theorem \ref{thm1}. Using similar argument, we prove a Michael-Simon type Sobolev inequality (Theorem \ref{thm2}) for manifolds with asymptotic nonnegative $k$-Ricci curvature in Section \ref{sec4}.

\section{Preliminaries}\label{sec2}
\subsection{$k$-Ricci curvature}
Let $M$ be a Riemannian manifold. The \emph{$k$-Ricci curvature} is the average of sectional curvature over a $k$-dimensional subspace of the tangent space. Let $X\in T_p M$ be a unit tangent vector and $V\subset T_p M$ be a $k$-dimensional subspace such that $X\perp V$. The $k$-Ricci curvature of $(X,V)$ is defined by
$$
Ric_k(X,V)=\frac{1}{k}\sum_{i=1}^k \langle R(e_i,X)e_i,X\rangle,
$$
where $R$ is the Riemann curvature tensor and $\{e_i\}$ is an orthonormal basis of $V$.
It is worth noting that $Ric_1$ is just sectional curvature and $Ric_{n-1}$ is just Ricci curvature. We say a manifold $M$ has nonnegative $k$-Ricci curvature (denoted by $Ric_k\ge 0$) if at each point $p\in M$, for any unit tangent vector $X\in T_p M$ and $k$-dimensional subspace $V$ such that $X\perp V$, we have $Ric_k(X,V)\ge 0$. 
\begin{rem} 
There is another definition for the intermediate Ricci curvature, say $\widetilde{Ric}_k(X,V)$, without the restriction $X\bot V$ (e.g. \cite{KM18, wang22}). We remark that $\widetilde{Ric}_k \ge 0$ is equivalent to $Ric_{k-1} \ge 0$ for each $k\ge 2$.
\end{rem}
\subsection{Asymptotic $k$-Ricci curvature}
The notion of asymptotically nonnegative curvature was first introduced by Abresch \cite{Abresch85}. Let $\lambda:[0,+\infty)\to[0,+\infty)$ be a nonnegative nonincreasing continuous function satisfying
\begin{equation}\label{defb0}
b_0:=\int_0^{+\infty}s\lambda(s)\textrm{d}s<+\infty,
\end{equation}
and
\begin{equation}\label{defb1}
b_1:=\int_0^{+\infty}\lambda(s)\textrm{d}s<+\infty.
\end{equation}
A complete noncompact $n$-dimensional Riemannian manifold $(M,g)$ is said to have asymptotically nonnegative Ricci curvature (sectional curvature, respectively) if there is a base point $o\in M$ such that
$$
Ric(\cdot,\cdot)\ge-(n-1)\lambda(d(o,q))g\quad (Sec\ge-\lambda(d(o,q)), \quad \textrm{respectively}),
$$
at each point $q\in M$. Similarly, we can define the concept of asymptotically nonnegative $k$-Ricci curvature in the sense that there exists a base point $o\in M$ such that 
$$
Ric_k\ge -k\lambda(d(o,q))
$$
at each point $q\in M$. By definition, a manifold whose (Ricci, sectional or $k$-Ricci, respectively) curvature is either nonnegative outside a compact domain or asymptotically flat has asymptotically nonnegative (Ricci, sectional or $k$-Ricci, respectively) curvature.

\subsection{Asymptotic volume ratio}
Let $M$ be a complete noncompact $n$-dimensional Riemannian manifold. The \emph{asymptotic volume ratio} $\theta$ can be regarded as the ratio of the volume of geodesic ball in $M$ to the volume of Euclidean ball in $\mathbb{R}^n$ with same, arbitrary large radius. Precisely, the asymptotic volume ratio $\theta$ is defined as
$$
\theta := \lim_{r\to\infty}\frac{|\{p\in M: d(p,q)<r\}|}{|B^n|r^n},
$$
where $q$ is an arbitrary fixed point in $M$ and $B^n$ is the unit ball in $\mathbb{R}^{n+1}$.
If $M$ has nonnegative Ricci curvature, the Bishop-Gromov volume comparison theorem indicates that $\theta$ exists and $\theta\le 1$.

Similarly, let $h(t)$ be the unique solution of
\begin{equation}
    \left\{\begin{array}{ll}
         h''(t)=\lambda(t)h(t),\\
         h(0)=0,\quad h'(0)=1,
    \end{array}
    \right.
\end{equation}
where $\lambda$ is the nonnegative function given in Section 2.2. The \emph{asymptotic volume ratio of $M$ with respect to $h$} is defined by
\begin{equation}
    \theta_h := \lim_{r\to+\infty}\frac{|\{q\in M:d(o,q)<r\}|}{n|B^n|\int_0^r h^{n-1}(t)\textrm{d}t}.
\end{equation}

\section{Manifolds with nonnegative $k$-Ricci curvature}\label{sec3}
In this section, we assume that $(M,g)$ is a complete noncompact $(n+m)$-dimensional Riemannian manifold with nonnegative $k$-Ricci curvature, where $k=\min(n-1,m-1)$. We also assume that $\Sigma$ is a compact $n$-dimensional submanifold of $M$ (possibly with boundary $\partial\Sigma$), and $f$ is a positive smooth function on $\Sigma$. Let $\bar D$ and $D^\Sigma$ denote the Levi-Civita connection on $(M,g)$ and the induced connection on $\Sigma$, respectively. Let $\bar R$ denote the Riemann curvature tensor on $(M,g)$. For any tangent vector fields $X$, $Y$ on $\Sigma$ and normal vector field $\eta$ along $\Sigma$, the second fundamental form $II$ of $\Sigma$ is given by 
$$
\langle  II(X,Y),\eta\rangle =\langle \bar D_X Y, \eta\rangle.
$$

We only need to prove Theorem \ref{thm1} in the case that $\Sigma$ is connected. By scaling, we assume that
\begin{equation}\label{scaling}
\int_\Sigma \sqrt{|D^\Sigma f|^2+f^2 |H|^2}+\int_{\partial\Sigma}f = n\int_\Sigma f^{\frac{n}{n-1}}.
\end{equation}

Since $\Sigma$ is connected, there exists a solution $u$ to the following Neumann boundary problem.
\begin{equation}
\left\{\begin{array}{lll}
    \mathrm{div}_\Sigma(fD^\Sigma u) = n f^{\frac{n}{n-1}}-\sqrt{|D^\Sigma f|^2+f^2 |H|^2}, &\textrm{in}\, \Sigma,\\
    \\
    \langle D^\Sigma u, \nu \rangle =1, &\textrm{on}\, \partial\Sigma,
\end{array}\right.
\end{equation}
where $\nu$ is the outward conormal to $\partial\Sigma$. By standard elliptic regularity theory, $u\in C^{2,\gamma}(\Sigma)$ for each $0<\gamma<1$ (cf. Theorem 6.30 in \cite{GTbook}).

As in \cite{Brendle22}, we define
\begin{equation*}
\begin{aligned}
    \Omega :=& \{x\in\Sigma\backslash\partial\Sigma : |D^\Sigma u(x)|<1\},\\
    U:=&\{(x,y):x\in\Sigma\backslash\partial\Sigma, y\in T_x^{\perp}\Sigma,|D^\Sigma u(x)|^2+|y|^2 < 1\}.
\end{aligned}
\end{equation*}
For each $r>0$, we denote by $A_r$ the contact set, that is the set of all points $(\bar x,\bar y)\in U$ with the property that
$$
r u(x)+\frac{1}{2} d(x,\exp_{\bar x}(rD^\Sigma u(\bar x)+r \bar y))^2 \ge r u (\bar x)+\frac{1}{2} r^2(|D^\Sigma u(x)|^2+|y|^2),
$$
for all $x$ in $\Sigma$. Moreover, for each $r>0$ we define the transport map $\Phi_r: T^\perp\Sigma\to M$ by
$$
\Phi_r(x,y)=\exp_x(rD^\Sigma u(x)+r y)
$$
for all $x\in\Sigma$ and $y\in T^\perp_x \Sigma$. For each $0<\gamma<1$, since $u$ is of class $C^{2,\gamma}$, $\Phi_r$ is of class $C^{1,\gamma}$.

The following four lemmas are due to Brendle \cite{Brendle22}. Their proofs are independent of the curvature condition, so they also hold in our setting and we omit the proofs here.
\begin{lem}[Lemma 4.1 in \cite{Brendle22}]\label{lemma1}
Assume that $x\in\Omega$ and $y\in T_x^\perp\Sigma$ satisfy $|D^\Sigma u(x)|^2+|y|^2\le 1$. Then $\Delta_\Sigma u(x)-\langle H(x),y\rangle \le n f(x)^{\frac{1}{n-1}}$.
\end{lem}

\begin{lem}[Lemma 4.2 in \cite{Brendle22}]\label{lemma2}
For each $0\le\sigma<1$, the set 
$$
\{p\in M:\sigma r<d(x,p)<r \textrm{ for all } x\in\Sigma\}
$$
is contained in the set
$$
\Phi_r(\{(x,y)\in A_r : |D^\Sigma u(x)|^2+|y|^2 >\sigma^2\}).
$$
\end{lem}

\begin{lem}[Lemma 4.3 in \cite{Brendle22}]\label{lemma3}
Assume that $(\bar x, \bar y)\in A_r$, and let $\bar\gamma(t):=\exp_{\bar x}(tD^\Sigma u (\bar x)+t \bar y)$ for all $t\in[0,r]$. If $Z$ is a smooth vector field along $\bar\gamma$ satisfying $Z(0)\in T_{\bar x}\Sigma$ and $Z(r)=0$, then
\begin{equation*}
    \begin{aligned}
        &(\nabla^2_\Sigma u)(Z(0),Z(0))-\langle  II(Z(0),Z(0)),\bar y\rangle\\
        &+\int_0^r(|\bar D_t Z(t)|^2-\bar R(\bar\gamma'(t),Z(t),\bar\gamma'(t),Z(t)))\textrm{d}t\ge0.
    \end{aligned}
\end{equation*}
\end{lem}

\begin{lem}[Lemma 4.5 in \cite{Brendle22}]\label{lemma4}
Assume that $(\bar x, \bar y)\in A_r$, and let $\bar\gamma(t):=\exp_{\bar x}(tD^\Sigma u (\bar x)+t \bar y)$ for all $t\in[0,r]$. Moreover, let $\{e_1,\dots,e_n\}$ be an orthonormal basis of $T_{\bar x}\Sigma$. Suppose that $W$ is a Jacobi field along $\bar{\gamma}$ satisfying $W(0)\in T_{\bar x}\Sigma$ and $\langle \bar D_t W(0), e_j\rangle=(D^2_\Sigma u)(W(0),e_j)-\langle  II(W(0),e_j),\bar y\rangle$ for each $1\le j\le n$. IF $W(\tau)=0$ for some $\tau\in(0,r)$, then $W$ vanishes identically.
\end{lem}

Now with the preparation above, we are in a position to prove Theorem \ref{thm1}.

\begin{proof}[Proof of Theorem \ref{thm1}]
Throughout the proof we use the following notions of indices
$$
1\le i,j\le n,\qquad n+1\le\alpha,\beta\le n+m,\qquad 1\le A,B \le n+m.
$$
For any $r>0$ and $(\bar x, \bar y)\in A_r$, let $\{e_i\}_{1\le i\le n}$ be any given orthonormal basis in $T_{\bar x}\Sigma$. We can choose geodesic normal coordinates $(x_1,\dots,x_n)$ on $\Sigma$ around $\bar x$ such that $\frac{\partial}{\partial x_i}=e_i$ at $\bar x$ for each $1\le i\le n$. 
Let $\{e_\alpha\}_{n+1\le \alpha\le n+m}$ be a local orthonormal frame of $T^\perp \Sigma$ around $\bar x$ such that $\langle \bar D_{e_i}e_\alpha, e_\beta\rangle=0$ at $\bar x$. 
Now a normal vector $y$ can be written as $y=\sum_\alpha y_\alpha e_\alpha$ and in this sense $(x_1,\dots,x_n,y_{n+1},\dots,y_{n+m})$ forms a local coordinate system on the total space of the normal bundle $T^\perp\Sigma$.

Let $\bar\gamma(t):=\exp_{\bar x}(tD^\Sigma u (\bar x)+t \bar y)$ for all $t\in[0,r]$. For each $1\le A\le n+m$, let $E_A(t)$ be the parallel transport of $e_A$ along $\bar \gamma$. 
For each $1\le i \le n$, let $X_i(t)$ be the unique Jacobi field along $\bar \gamma$ satisfying 
\begin{equation*}
\begin{aligned}
X_i(0)&=e_i,\\
\langle \bar D_t X_i(0),e_j\rangle&=(D^2_\Sigma u)(e_i,e_j)-\langle II(e_i,e_j),\bar y\rangle,\\
\langle \bar D_t X_i(0),e_\beta\rangle&=\langle II(e_i,D^\Sigma u),e_\beta\rangle.
\end{aligned}
\end{equation*}
For each $n+1\le \alpha\le n+m$, let $X_\alpha(t)$ be the unique Jacobi field along $\bar \gamma$ satisfying
\begin{equation*}
X_\alpha(0)=0,\qquad \bar D_t X_\alpha(0)=e_\alpha.
\end{equation*}
Lemma \ref{lemma4} tells us that $\{X_A(t)\}_{1\le A\le n+m}$ are linearly independent for each $t\in(0,r)$.

Then we define two $(n+m)\times(n+m)$-matrix $P(t)=(P_{AB}(t))$ and $S(t)=(S_{AB}(t))$ by
\begin{equation*}
    \begin{aligned}
        P_{AB}(t)&=\langle X_A(t),E_B(t)\rangle,\\
        S_{AB}(t)&=\bar R(\bar \gamma'(t),E_A(t),\bar \gamma'(t),E_B(t)).
    \end{aligned}
\end{equation*}
Based on the observation that 
\begin{equation}
    \frac{\partial\Phi_t}{\partial x_A}(\bar x, \bar y)=X_A(t)
\end{equation}
for each $1\le A\le n+m$, we conclude that 
$$
|\det \bar D\Phi_t(\bar x, \bar y)|=\det P(t)
$$
for all $t\in(0,r)$. Therefore, we only need to estimate $\det P(t)$.

By the definition of the Jacobi fields $X_A(t)$, the Jacobi equation reads
$$
P''(t)=-P(t)S(t)
$$
with the initial conditions
\begin{equation*}
P(0)=\left[
\begin{array}{cc}
    \delta_{ij} & 0  \\
    0 & 0
\end{array}
\right]
\end{equation*}
and
\begin{equation*}
P'(0)=\left[
\begin{array}{cc}
    (D^2_\Sigma u)(e_i,e_j)-\langle II(e_i,e_j),\bar y\rangle & \langle II(e_i,D^\Sigma u),e_\beta\rangle  \\
    0 & \delta_{\alpha\beta}
\end{array}
\right].
\end{equation*}

Moreover, since
$$
\frac{d}{dt}(P'(t)P(t)^T)=-P(t)S(t)P(t)^T+P'(t)P'(t)^T
$$
is symmetric for each $t\in(0,r)$, $P'(t)P(t)^T$ is also symmetric for each $t$. Let $Q(t)$ be a matrix defined by
$$
Q(t):=P(t)^{-1}P'(t)=P(t)^{-1}P'(t)P(t)^T(P(t)^{-1})^T,
$$
which is symmetric for each $t\in (0,r)$. 
Then the Riccati equation reads
\begin{equation}\label{Riccati1}
Q'(t)=-S(t)-Q^2(t).
\end{equation}

From the asymptotic expansion of $P(t)$
\begin{equation*}
P(t)=\left[
\begin{array}{cc}
    \delta_{ij}+O(t) & O(t)  \\
    O(t) & t\delta_{\alpha\beta}+O(t^2)
\end{array}
\right],
\end{equation*}
a direct computation gives rise to
\begin{equation}\label{initialQ}
Q(t)=\left[
\begin{array}{cc}
    (D^2_\Sigma u)(e_i,e_j)-\langle II(e_i,e_j),\bar y\rangle +O(t) & O(1)  \\
    O(1) & t^{-1}\delta_{\alpha\beta}+O(1)
\end{array}
\right]
\end{equation}
as $t\to 0$.

By taking partial trace of $Q(t)$, we reduce the Riccati equation \eqref{Riccati1} to the following two equations
\begin{equation}\label{Riccati2}
\begin{aligned}
\sum_i Q'_{ii}(t)+\frac{1}{n}(\sum_i Q_{ii}(t))^2 &\le \sum_i Q'_{ii}(t)+\sum_i(Q^2)_{ii}(t)=-\sum_i S_{ii}(t),\\
\sum_\alpha Q'_{\alpha\alpha}(t)+\frac{1}{m}(\sum_\alpha Q_{\alpha\alpha}(t))^2 &\le \sum_\alpha Q'_{\alpha\alpha}(t)+\sum_\alpha(Q^2)_{\alpha\alpha}(t)=-\sum_\alpha S_{\alpha\alpha}(t),
\end{aligned}
\end{equation}
where we use the Cauchy-Schwarz inequality.

By assumption, $M$ has nonnegative $k$-Ricci curvature. We claim that $\sum_i S_{ii}(t)$ and $\sum_\alpha S_{\alpha \alpha}(t)$ are nonnegative for any $t\in[0,r)$.

In fact, for fixed $t\in[0,r)$, without loss of generality we can choose $e_1=\frac{D^{\Sigma}u(\bar x)}{|D^{\Sigma}u(\bar x)|}$ and $e_{n+1}=\frac{\bar y}{|\bar y|}$. Denote $a:=\sqrt{|D^{\Sigma}u(\bar x)|^2+|\bar y|^2}$, then there exists an angle $s$ and a vector field $\xi$ along $\bar{\gamma}$, such that
\begin{equation}
                                  \left(\begin{array}{cc}
                                  \frac{1}{a}\bar{\gamma}'(t)   \\
                                  \frac{1}{a}\xi(t)  \\
                                  \end{array}\right)=
                                  \left( \begin{array}{cc}
                                  \cos s& \sin s  \\
                                  -\sin s & \cos s \\
                                  \end{array} \right)
                                  \left(\begin{array}{cc}
                                  E_1(t)  \\
                                  E_{n+1}(t) \\
                                  \end{array}\right),
\end{equation}
as well as
\begin{equation}
                                  \left(\begin{array}{cc}
                                  E_1(t)  \\
                                  E_{n+1}(t) \\
                                  \end{array}\right)=
                                  \left( \begin{array}{cc}
                                  \cos s& -\sin s  \\
                                  \sin s & \cos s \\
                                  \end{array} \right)
                                  \left(\begin{array}{cc}
                                 \frac{1}{a}\bar{\gamma}'(t)   \\
                                  \frac{1}{a}\xi(t)  \\
                                  \end{array}\right).
\end{equation}
Now
\begin{equation}
\begin{aligned}
S_{11}(t)&=\bar R(\bar \gamma'(t),E_1(t),\bar \gamma'(t),E_1(t))\\
&=\bar R(\bar \gamma'(t),\cos s \cdot \frac{1}{a}\bar{\gamma}'(t) - \sin s \cdot \frac{1}{a}\xi(t) ,\bar \gamma'(t), \cos s \cdot \frac{1}{a} \bar{\gamma}'(t) - \sin s\cdot \frac{1}{a} \xi(t))\\
&=\sin^2 s\bar R(\bar \gamma'(t), \frac{1}{a} \xi(t),\bar \gamma'(t), \frac{1}{a} \xi(t)).
\end{aligned}
\end{equation}
Since $\{\frac{1}{a}\bar \gamma'(t), E_2(t), \dots, E_n(t), \frac{1}{a}\xi(t), E_{n+2}(t), \dots, E_{n+m}(t)\}$ form an orthogonal basis, we can compute that
\begin{equation}\label{sum:sii}
\begin{aligned}
\sum_{i=1}^{n}S_{ii}(t)&=\sum_{i=1}^{n}\bar R(\bar \gamma'(t),E_i(t),\bar \gamma'(t),E_i(t))\\
&=\bar R(\bar \gamma'(t),E_1(t),\bar \gamma'(t),E_1(t))+\sum_{i=2}^{n}\bar R(\bar \gamma'(t),E_i(t),\bar \gamma'(t),E_i(t))\\
&=\sin^2 s\bar R(\bar \gamma'(t), \frac{1}{a} \xi(t),\bar \gamma'(t), \frac{1}{a} \xi(t))+\sum_{i=2}^{n}\bar R(\bar \gamma'(t),E_i(t),\bar \gamma'(t),E_i(t))\\
&= \sin^2 s \left(\bar R(\bar \gamma'(t), \frac{1}{a} \xi(t),\bar \gamma'(t), \frac{1}{a} \xi(t))+\sum_{i=2}^{n}\bar R(\bar \gamma'(t),E_i(t),\bar \gamma'(t),E_i(t))\right)\\
&\quad+\cos^2 s Ric_{n-1}\left(\bar \gamma'(t), \mathrm{span}\{E_2(t), \dots, E_n(t)\}\right)\\
&\ge \sin^2 s \left(\bar R(\bar \gamma'(t), \frac{1}{a} \xi(t),\bar \gamma'(t), \frac{1}{a} \xi(t))+\sum_{i=2}^{n}\bar R(\bar \gamma'(t),E_i(t),\bar \gamma'(t),E_i(t))\right)\\
&=\sin^2 s Ric_{n}\left(\bar \gamma'(t), \mathrm{span}\{\frac{1}{a} \xi(t),E_2(t), \dots, E_n(t)\}\right)\\
&\ge0.
\end{aligned}
\end{equation}
Analogously we also have $\sum_\alpha S_{\alpha \alpha}(t)\ge0$.

Therefore, the equations \eqref{Riccati2} become
\begin{equation}\label{Riccati3}
\begin{aligned}
\sum_i Q'_{ii}(t)+\frac{1}{n}(\sum_i Q_{ii}(t))^2 &\le 0,\\
\sum_\alpha Q'_{\alpha\alpha}(t)+\frac{1}{m}(\sum_\alpha Q_{\alpha\alpha}(t))^2 &\le 0.
\end{aligned}
\end{equation}

A standard ODE comparison gives
\begin{equation}\label{logdetP1}
\begin{aligned}
    \frac{d}{dt}\log \det P(t)&=\Tr Q(t)\\
    &=\sum_i Q_{ii}(t)+\sum_\alpha Q_{\alpha\alpha}(t)\\
    &\le \frac{n(\Delta_\Sigma u(\bar x)-\langle H(\bar x),\bar y\rangle)}{t(\Delta_\Sigma u(\bar x)-\langle H(\bar x),\bar y\rangle)+n}+\frac{m}{t}.
    \end{aligned}
\end{equation}

Integrating \eqref{logdetP1} over $[\epsilon,t]$ for $0<\epsilon<t$ and letting $\epsilon\to 0^+$, we obtain
\begin{equation}\label{detP1}
    \begin{aligned}
        &|\det \bar D\Phi_t(\bar x, \bar y)|=\det P(t)\\
        &\le\left(1+\frac{t}{n}(\Delta_\Sigma u(\bar x)-\langle H(\bar x),\bar y\rangle)\right)^nt^m,
    \end{aligned}
\end{equation}
for all $t\in(0,r)$.

By Lemma \ref{lemma1}, we have
\begin{equation}
    |\det \bar D\Phi_r(\bar x, \bar y)|\le r^m(1+r f(\bar x)^{\frac{1}{n-1}})^n.
\end{equation}

As in the proof of \cite{Brendle22}, together with the above estimate, Lemma \ref{lemma2} tells us
\begin{equation}
\begin{aligned}
    &|\{p\in M:\sigma r<d(x,p)<r \textrm{ for all } x\in\Sigma\}|\\
    &\le \int_\Omega\left(\int_{\{y\in T_x^\perp \Sigma:\sigma^2<|D^\Sigma u(x)|^2+|y|^2<1\}}|\det \bar D\Phi_r(x, y)| 1_{A_r}(x,y) \textrm{d}y\right)\textrm{d} vol(x)\\
    &\le \int_\Omega\left(\int_{\{y\in T_x^\perp \Sigma:\sigma^2<|D^\Sigma u(x)|^2+|y|^2<1\}} r^m (1+r f(x)^{\frac{1}{n-1}})^n\textrm{d}y\right)\textrm{d}vol(x)\\
    &\le\frac{m}{2}|B^m|(1-\sigma^2)\int_\Omega r^m (1+r f(x)^{\frac{1}{n-1}})^n\textrm{d}vol(x),
\end{aligned}
\end{equation}
for all $r>0$ and all $0\le\sigma<1$.
Now dividing both sides by $r^{n+m}$ and letting $r\to\infty$, we have
\begin{equation}
    |B^{n+m}|(1-\sigma^{n+m})\theta\le \frac{m}{2}|B^m|(1-\sigma^2)\int_\Omega f(x)^{\frac{n}{n-1}}\textrm{d}vol(x),
\end{equation}
for all $0\le\sigma<1$. Finally, dividing both sides by $1-\sigma$ and letting $\sigma\to1$, we have
\begin{equation}
    (n+m)|B^{n+m}|\theta\le m|B^m| \int_\Omega f(x)^{\frac{n}{n-1}}\textrm{d}vol(x) \le m|B^m| \int_\Sigma f(x)^{\frac{n}{n-1}}\textrm{d}vol(x)
\end{equation}
Consequently, it is apparent from the scaling assumption \eqref{scaling} that
\begin{equation}
    \begin{aligned}
        &\int_\Sigma \sqrt{|D^\Sigma f|^2+f^2 |H|^2}+\int_{\partial\Sigma}f = n\int_\Sigma f^{\frac{n}{n-1}}\\
        &\ge n \left(\frac{(n+m)|B^{n+m}|}{m|B^m|}\right)^{\frac{1}{n}}\theta^{\frac{1}{n}}\left(\int_\Sigma f^{\frac{n}{n-1}}\right)^{\frac{n-1}{n}}.
    \end{aligned}
\end{equation}

\end{proof}

\section{Manifolds with asymptotically nonnegative $k$-Ricci curvature}\label{sec4}

\begin{proof}[Proof of Theorem \ref{thm2}]
By scaling, we assume that
\begin{equation}\label{scaling2}
\int_\Sigma \sqrt{|D^\Sigma f|^2+f^2 |H|^2}+2n b_1\int_\Sigma f + \int_{\partial\Sigma}f = n\int_\Sigma f^{\frac{n}{n-1}}.
\end{equation}

We use the same notions as in the proof of Theorem \ref{thm1} except that $u$ is a solution of the following Neumann boundary problem
\begin{equation}
\left\{\begin{array}{lll}
    \mathrm{div}_\Sigma(fD^\Sigma u) = n f^{\frac{n}{n-1}}-\sqrt{|D^\Sigma f|^2+f^2 |H|^2}-2n b_1 f, &\textrm{in}\, \Sigma,\\
    \\
    \langle D^\Sigma u, \nu \rangle =1, &\textrm{on}\, \partial\Sigma.
\end{array}\right.
\end{equation}

Since we only change the definition of $u$, Lemma \ref{lemma2}, Lemma \ref{lemma3} and Lemma \ref{lemma4} still hold. We also need another version of Lemma \ref{lemma1}.
\begin{lem}[Lemma 3.1 in \cite{DLL22}]\label{lemma1'}
    Assume that $x\in\Omega$ and $y\in T_x^\perp \Sigma$ satisfy $|D^\Sigma u(x)|^2+|y|^2\le 1$. Then $\Delta_\Sigma u(x)-\langle H(x),y\rangle\le n f(x)^{\frac{1}{n-1}}-2n b_1$.
\end{lem}

By assumption, $M$ has asymptotically nonnegative $k$-Ricci curvature with respect to a base point $o$ in $M$. Analogous to the computation \eqref{sum:sii}, we have 
$$
\sum_i S_{ii}(t)\ge (\cos^2 s-n) \lambda(d(o,\bar \gamma(t)))
$$
and
$$\sum_\alpha S_{\alpha \alpha}(t)\ge (\sin^2 s-m) \lambda(d(o,\bar \gamma(t))),
$$
where $s$ is the angle between $\bar \gamma'(t)$ and $E_1(t)$. 

Let $\phi$ and $\widetilde \phi$ be defined by
\begin{equation}
\begin{aligned}
    \phi:=&e^{\frac{1}{n}\int_0^t\sum_i Q_{ii}(\tau)\textrm{d}\tau},\\
    \widetilde{\phi}:=&t e^{\frac{1}{m}\int_0^t\sum_\alpha(Q_{\alpha\alpha}(\tau)-\frac{1}{\tau})\textrm{d}\tau},
\end{aligned}
\end{equation}
respectively.
Then with the initial condition \eqref{initialQ}, \eqref{Riccati2} reduce to
\begin{equation}\label{ODEphi}
    \begin{aligned}
        \left\{
        \begin{array}{lll}
            \phi''\le-\frac{1}{n}\sum_i S_{ii}(t)\phi \le\frac{n-\cos^2 s}{n} \lambda(d(o,\bar \gamma(t))) \phi,\\
            \\
            \phi(0)=1,\quad\phi'(0)=\frac{1}{n}(\Delta_\Sigma u(\bar x)-\langle H(\bar x),\bar y\rangle),
        \end{array}
        \right.
    \end{aligned}
\end{equation}
and
\begin{equation}\label{ODEtphi}
    \begin{aligned}
        \left\{
        \begin{array}{lll}
            \widetilde \phi''\le-\frac{1}{m}\sum_\alpha S_{\alpha\alpha}(t)\widetilde \phi \le \frac{m-\sin^2 s}{m} \lambda(d(o,\bar \gamma(t))) \widetilde \phi ,\\
            \\
            \widetilde \phi(0)=0,\quad\widetilde\phi'(0)=1.
        \end{array}
        \right.
    \end{aligned} 
\end{equation}

Next, in order to estimate $\phi$ and $\widetilde \phi$, we should compare them with standard ODE solutions. Let $\psi_1$ and $\psi_2$ be solutions to the following ODEs 

\begin{equation}\label{ODEpsi1}
    \begin{aligned}
        \left\{
        \begin{array}{lll}
            \psi_1''=\frac{n-\cos^2 s}{n} \lambda(d(o,\bar \gamma(t)))\psi_1,\\
            \\
            \psi_1(0)=0,\quad\psi_1'(0)=1,
        \end{array}
        \right.
    \end{aligned} 
\end{equation}

\begin{equation}\label{ODEpsi2}
    \begin{aligned}
        \left\{
        \begin{array}{lll}
            \psi_2''=\frac{n-\cos^2 s}{n} \lambda(d(o,\bar \gamma(t)))\psi_2,\\
            \\
            \psi_2(0)=1,\quad\psi_2'(0)=0,
        \end{array}
        \right.
    \end{aligned} 
\end{equation}
respectively. Then the function $\psi:=\psi_1(\frac{1}{n}(\Delta_\Sigma u(\bar x)-\langle H(\bar x),\bar y\rangle))+\psi_2$ is the solution to
\begin{equation}\label{ODEpsi}
    \begin{aligned}
        \left\{
        \begin{array}{lll}
            \psi''=\frac{n-\cos^2 s}{n} \lambda(d(o,\bar \gamma(t)))\psi,\\
            \\
            \psi(0)=1,\quad\psi'(0)=\frac{1}{n}(\Delta_\Sigma u(\bar x)-\langle H(\bar x),\bar y\rangle).
        \end{array}
        \right.
    \end{aligned} 
\end{equation}
By a standard ODE comparison result (cf. \cite{CW06}, Lemma 2.4A), we have
\begin{equation}\label{comparison1}
    \frac{\phi'}{\phi}\le\frac{\psi'}{\psi}.
\end{equation}
Similarly, let $\widetilde\psi$ be the solution to
\begin{equation}\label{ODEtpsi}
    \begin{aligned}
        \left\{
        \begin{array}{lll}
            \widetilde\psi''=\frac{m-\sin^2 s}{m} \lambda(d(o,\bar \gamma(t)))\widetilde\psi,\\
            \\
            \widetilde\psi(0)=0,\quad\widetilde\psi'(0)=1.
        \end{array}
        \right.
    \end{aligned} 
\end{equation}
Then the standard ODE comparison (cf. \cite{CW06}, Lemma 2.4A) gives
\begin{equation}\label{comparison2}
    \frac{\widetilde\phi'}{\widetilde\phi}\le\frac{\widetilde\psi'}{\widetilde\psi}.
\end{equation}

Now taking \eqref{comparison1}, \eqref{comparison2} into consideration, by definition of $\phi$ and  $\widetilde\phi$, we have
\begin{equation}\label{logdetP}
    \frac{d}{dt}\log \det P(t)=\Tr Q(t)=n\frac{\phi'}{\phi}+m\frac{\widetilde\phi'}{\widetilde\phi} \le n\frac{\psi'}{\psi}+m\frac{\widetilde\psi'}{\widetilde\psi},
\end{equation}
for all $t\in(0,r)$.
Integrating \eqref{logdetP} over $[\epsilon,t]$ for $0<\epsilon<t$ and letting $\epsilon\to 0^+$, we obtain
\begin{equation}\label{detP}
    \begin{aligned}
        &|\det \bar D\Phi_t(\bar x, \bar y)|=\det P(t)\\
        &\le\psi(t)^n\widetilde\psi(t)^m=\left[\frac{\psi_2}{\psi_1}(t)+\frac{1}{n}(\Delta_\Sigma u(\bar x)-\langle H(\bar x),\bar y\rangle)\right]^n\psi_1(t)^n\widetilde\psi(t)^m,
    \end{aligned}
\end{equation}
for all $t\in(0,r)$.
Moreover, a standard ODE comparison of \eqref{ODEpsi1}, \eqref{ODEpsi2} and \eqref{ODEtpsi} gives
\begin{equation}\label{comparison3}
\begin{aligned}
    \psi_1(t)&\le t e^{\int_0^t \tau(\frac{n-\cos^2 s}{n} \lambda(d(o,\bar \gamma(\tau))))\textrm{d}\tau},\\
    \widetilde\psi(t)&\le t e^{\int_0^t \tau(\frac{m-\sin^2 s}{m} \lambda(d(o,\bar \gamma(\tau))))\textrm{d}\tau},\\
    \frac{\psi_2}{\psi_1}(t)&\le \int_0^t\left(\frac{n-\cos^2 s}{n} \lambda(d(o,\bar \gamma(\tau)))\right)\textrm{d}\tau+\frac{1}{t}.
\end{aligned}
\end{equation}

Furthermore, we have
\begin{equation}\label{comparison4}
        \int_0^t\lambda(d(o,\bar \gamma(\tau)))\textrm{d}\tau\le 2 b_1
\end{equation}
and
\begin{equation}\label{comparison5}
    \begin{aligned}
	\int_0^t \tau\lambda(d(o,\bar \gamma(\tau)))\textrm{d}\tau&\le\int_0^t \tau\lambda(|d(o,\bar x)-d(\bar x,\bar\gamma(\tau))|)\textrm{d}\tau\\
        &\le\int_0^t \tau\lambda(|d(o,\bar x)-\tau|)\textrm{d}\tau\\
        &\le2r_0 b_1+b_0.
    \end{aligned}
\end{equation}

By substituting \eqref{comparison3}, \eqref{comparison4} and \eqref{comparison5} into \eqref{detP}, together with Lemma \ref{lemma1'} we have
\begin{equation}
\begin{aligned}
    &|\det \bar D\Phi_t(\bar x, \bar y)|=\det P(t)\\
    &\le \left[2 b_1+\frac{1}{t}+\frac{1}{n}(\Delta_\Sigma u(\bar x) -\langle H(\bar x), \bar y\rangle)\right]^n t^{n+m}e^{(n+m-1)(2r_0b_1+b_0)}\\
    &\le t^m(1+t f(\bar x)^{\frac{1}{n-1}})^n e^{(n+m-1)(2r_0 b_1+b_0)}.
\end{aligned}
\end{equation}
Analogous to the proof of Theorem \ref{thm1}, Lemma \ref{lemma2} tells us 
\begin{equation}
\begin{aligned}
    &|\{p\in M:\sigma r<d(x,p)<r \textrm{ for all } x\in\Sigma\}|\\
    &\le\frac{m}{2}|B^m|(1-\sigma^2)e^{(n+m-1)(2r_0b_1+b_0)}\int_\Omega r^m (1+r f(x)^{\frac{1}{n-1}})^n\textrm{d}vol(x).
\end{aligned}
\end{equation}

Now dividing both sides by $(n+m)\int_0^r h(t)^{n+m-1}\textrm{d}t$ and letting $r\to\infty$, we have
\begin{equation}
    \begin{aligned}
        &|B^{n+m}|(1-\sigma^{n+m})\theta_h\\
        &\le\frac{m}{2}|B^m|(1-\sigma^2)e^{(n+m-1)(2r_0b_1+b_0)}\int_\Omega f(x)^{\frac{n}{n-1}}\textrm{d}vol(x)\lim_{r\to\infty}\frac{r^{n+m}}{(n+m)\int_0^r h(t)^{n+m-1}\textrm{d}t}\\
        &\le\frac{m}{2}|B^m|(1-\sigma^2)\left(\frac{e^{2r_0 b_1+b_0}}{1+b_0}\right)^{n+m-1}\int_\Omega f(x)^{\frac{n}{n-1}}\textrm{d}vol(x).\\  \end{aligned}
\end{equation}
Here we use the fact that $h(t)\ge t$ and $\underline\lim_{t\to \infty} h'(t)\ge 1+b_0$ in the last line.
Then dividing both sides by $1-\sigma$ and letting $\sigma\to 1$, we have
\begin{equation}
\begin{aligned}
(n+m)|B^{n+m}|\theta_h&\le m|B^m|\left(\frac{e^{2r_0 b_1+b_0}}{1+b_0}\right)^{n+m-1}\int_\Omega f(x)^{\frac{n}{n-1}}\textrm{d}vol(x)\\
&\le m|B^m|\left(\frac{e^{2r_0 b_1+b_0}}{1+b_0}\right)^{n+m-1}\int_\Sigma f(x)^{\frac{n}{n-1}}\textrm{d}vol(x).
\end{aligned}
\end{equation}
Finally, it follows from the scaling assumption \eqref{scaling2} that
\begin{equation}
    \begin{aligned}
        &\int_\Sigma \sqrt{|D^\Sigma f|^2+f^2 |H|^2}+2n b_1\int_\Sigma f + \int_{\partial\Sigma}f = n\int_\Sigma f^{\frac{n}{n-1}}\\
        &\ge n \left(\frac{(n+m)|B^{n+m}|}{m|B^m|}\right)^{\frac{1}{n}}\theta_h^{\frac{1}{n}}\left(\frac{1+b_0}{e^{2r_0b_1+b_0}}\right)^{\frac{n+m-1}{n}}\left(\int_\Sigma f^{\frac{n}{n-1}}\right)^{\frac{n-1}{n}}.
    \end{aligned}
\end{equation}
\end{proof}
\begin{ack}
The authors were partially supported by the National Natural Science Foundation of China under grants No.~11831005 and No.~12061131014. 
They would like to thank Professor Chao Qian for his helpful suggestions. They also thank Kai-Hsiang Wang for bringing their attention to \cite{KM18, wang22} and for helpful comments on the earlier version of this paper.
\end{ack}

\bibliographystyle{abbrv}
\bibliography{refs}
\end{document}